\documentclass[xcolor=dvipsnames,svgnames,table,reqno]{amsart}

\input xy
\xyoption{all}
\usepackage{accents}
\usepackage{epsfig}
\usepackage{xcolor}
\usepackage{amsthm}
\usepackage{amssymb}
\usepackage{bbm}
\usepackage{amsmath}
\usepackage{amscd}
\usepackage{amsopn}
\usepackage{graphicx}
\usepackage{xspace}

\usepackage{hhline}
\usepackage{easybmat}
\usepackage{caption}   
\usepackage{relsize}

\usepackage{url}
\usepackage{enumitem, hyperref}\hypersetup{colorlinks}

\usepackage{stmaryrd}

\usepackage[T1]{fontenc}

\colorlet{purpleB70}{blue!70!red}

\colorlet{orangeR65}{red!65!yellow}

\definecolor{red2}{HTML}{d41173}

\definecolor{neongreen}{HTML}{1bf702}

\definecolor{radicalred}{HTML}{FF355E}

\definecolor{denim}{HTML}{1560BD}

\definecolor{darkcyan}{rgb}{0.0, 0.55, 0.55}

\definecolor{cilek}{HTML}{FF43A4}

\definecolor{mor}{HTML}{9F00C5}

\definecolor{phlox}{rgb}{0.87, 0.0, 1.0}

\definecolor{fluorescentpink}{HTML}{FF1493}

\definecolor{napiergreen}{rgb}{0.16, 0.5, 0.0}

\definecolor{kellygreen}{rgb}{0.3, 0.73, 0.09}

\definecolor{parisgreen}{HTML}{ 50C878 }

\definecolor{palatinateblue}{rgb}{0.15, 0.23, 0.89}

\definecolor{ceruleanblue}{rgb}{0.16, 0.32, 0.75}

\definecolor{brandeisblue}{rgb}{0.0, 0.44, 1.0}

\definecolor{KLMblue}{HTML}{0FC0FC}

\definecolor{cinnamon}{rgb}{0.82, 0.41, 0.12}

\definecolor{darkorange}{rgb}{1.0, 0.55, 0.0}

\definecolor{darktangerine}{rgb}{1.0, 0.66, 0.07}

\definecolor{deepcarrotorange}{rgb}{0.91, 0.41, 0.17}

\definecolor{internationalorange}{HTML}{FF4F00}

\definecolor{persimmon}{HTML}{EC5800}

\definecolor{pumpkin}{HTML}{FF7518}

\definecolor{darkred}{rgb}{1,0,0} 
\definecolor{darkgreen}{rgb}{0,0.7,0}
\definecolor{darkblue}{rgb}{0,0,1}

\hypersetup{colorlinks,
linkcolor=darkblue,
filecolor=darkgreen,
urlcolor=darkred,
citecolor=darkgreen}

\makeatletter
\def\reflb#1#2{\begingroup
    #2%
    \def\@currentlabel{#2}%
    \phantomsection\label{#1}\endgroup
}
\makeatother

\numberwithin{equation}{section}
\newtheorem{Theorem}{Theorem}
\numberwithin{Theorem}{section}

\newtheorem   {Lemma}[Theorem]{Lemma}
\newtheorem   {Claim}[Theorem]{Claim}

\newtheorem   {Proposition}[Theorem]{Proposition}
\newtheorem   {Corollary}[Theorem]{Corollary}
\theoremstyle {definition}
\newtheorem   {Definition}[Theorem]{Definition}
\theoremstyle {remark}
\newtheorem   {Remark}[Theorem]{Remark}
\newtheorem   {Example}[Theorem]{Example}

\newcommand{\CA}{{\mathcal A}}
\newcommand{\CC}{{\mathcal C}}

\newcommand{\tCC}{\tilde{\mathcal C}}
\newcommand{\tCB}{\tilde{\mathcal B}}

\newcommand{\id}{{\mathit id}}

\newcommand{\charr}{{\mathit char}\,}

\newcommand{\tJ}{\tilde{J}}

\newcommand{\hCB}{\hat{\mathcal B}}

\newcommand{\halpha}{\hat{\alpha}}
\newcommand{\hgamma}{\hat{\gamma}}
\newcommand{\hxi}{\hat{\xi}}
\newcommand{\heta}{\hat{\eta}}

\newcommand{\td}{\tilde{d}}

\newcommand{\tgamma}{\tilde{\gamma}}

\newcommand{\tvarphi}{\tilde{\varphi}}

\newcommand{\txi}{\tilde{\xi}}
\newcommand{\teta}{\tilde{\eta}}
\newcommand{\talpha}{\tilde{\alpha}}

\newcommand{\tGamma}{\tilde{\Gamma}}

\newcommand{\CB}{{\mathcal B}}

\newcommand{\PP}{{\mathcal P}}
\newcommand{\bPP}{\bar{\mathcal P}}

\def    \F      {{\mathbb F}}

\def    \R      {{\mathbb R}}

\def    \Z      {{\mathbb Z}}
\def    \N      {{\mathbb N}}
\def    \Q      {{\mathbb Q}}

\def    \CP     {{\mathbb C}{\mathbb P}}

\def    \12     {{\frac{1}{2}}}

\def    \rk     {\operatorname{rk}}

\def    \tSp    {\operatorname{\widetilde{Sp}}}

\def    \HF     {\operatorname{HF}}

\def    \Co     {\operatorname{CC}}
\def    \HQ     {\operatorname{HQ}}

\def    \H      {\operatorname{H}}
\def    \Co     {\operatorname{C}}

\def    \HM      {\operatorname{HM}}

\def    \CF      {\operatorname{CF}}

\def    \bPP     {\bar{\mathcal{P}}}
\def    \bx      {\bar{x}}
\def    \by      {\bar{y}}
\def    \bz      {\bar{z}}

\def    \Fl    {\scriptscriptstyle{Fl}}

\def    \qq    {\mathrm{q}}
\def    \hh    {\mathrm{h}}

\def    \cf    {\operatorname{c}}

\newcommand \Sq   {\operatorname{Sq}}

\newcommand \eq   {\scriptscriptstyle{eq}}
\newcommand \even {\scriptscriptstyle{even}}
\newcommand \unit  {\mathbbm{1}}

\begin{document}


\setlength{\smallskipamount}{6pt}
\setlength{\medskipamount}{10pt}
\setlength{\bigskipamount}{16pt}





\title[Another look at the Hofer--Zehnder conjecture]{Another look at
  the Hofer--Zehnder conjecture}

\author[Erman \c C\. inel\. i]{Erman \c C\. inel\. i}
\author[Viktor Ginzburg]{Viktor L. Ginzburg}
\author[Ba\c sak G\"urel]{Ba\c sak Z. G\"urel}

\address{E\c C and VG: Department of Mathematics, UC Santa Cruz, Santa
  Cruz, CA 95064, USA} \email{scineli@ucsc.edu}
\email{ginzburg@ucsc.edu}

\address{BG: Department of Mathematics, University of Central Florida,
  Orlando, FL 32816, USA} \email{basak.gurel@ucf.edu}



\date{\today} 

\thanks{The work is partially supported by NSF CAREER award
  DMS-1454342 (BG) and by Simons Foundation Collaboration Grant 581382
  (VG)}


\begin{abstract}
  We give a different and simpler proof of a slightly modified (and
  weaker) variant of a recent theorem of Shelukhin extending Franks'
  ``two-or-infinitely-many'' theorem to Hamiltonian diffeomorphisms in
  higher dimensions and establishing a sufficiently general case of
  the Hofer--Zehnder conjecture. A few ingredients of our proof are
  common with Shelukhin's original argument, the key of which is
  Seidel's equivariant pair-of-pants product, but the new proof
  highlights a different aspect of the periodic orbit dynamics of
  Hamiltonian diffeomorphisms.
\end{abstract}

\maketitle

\begin{center}\hfill\emph{Dedicated to Claude Viterbo on the occasion
    of his 60th birthday}
\end{center}

\tableofcontents

\section{Introduction and main results}

\subsection{Introduction}
\label{sec:intro}
In this paper we give a different and simpler proof of a slightly
modified and weaker version of a recent theorem of Shelukhin,
\cite{Sh:HZ}, extending Franks' ``two-or-infinitely-many'' theorem,
\cite{Fr1,Fr2}, to higher dimensions.

This celebrated theorem of Franks asserts that every area preserving
diffeomorphism of $S^2$ has either exactly two or infinitely many
periodic points. (Moreover, in the setting of Franks' theorem, there
are also strong growth rate results; see, e.g, \cite{FH, LeC, Ke12}.)
A generalization of Franks' theorem conjectured in \cite[p.\ 263]{HZ}
is that a Hamiltonian diffeomorphism $\varphi$ of a closed symplectic
manifold has infinitely many periodic points whenever $\varphi$ has
``more than absolutely necessary'' fixed points. (Hence, the title of
\cite{Sh:HZ} and of this paper.) The vaguely stated lower bound ``more
than absolutely necessary'' is usually interpreted as a lower bound
arising from some version of the Arnold conjecture, e.g., as the sum
of the Betti numbers. For $\CP^n$, the expected threshold is $n+1$
regardless of the non-degeneracy assumption and, in particular, it is
$2$ for $S^2=\CP^1$ as in Franks' theorem. A slightly different
interpretation of the conjecture, not directly involving the count of
fixed points, is that the presence of a fixed or periodic point that
is unnecessary from a homological or geometrical perspective is
already sufficient to force the existence of infinitely many periodic
points. We refer the reader to \cite{GG:hyperbolic, Gu:nc, Gu:hq} for
some results in this direction.

We note that whenever $\varphi$ has finitely many periodic points, by
passing to an iterate one can assume them to be fixed
points. Furthermore, $\varphi$ has infinitely many periodic points if
and only if it has infinitely many simple, i.e., un-iterated, periodic
orbits and the results are often stated in these terms. It is also
worth keeping in mind that all known Hamiltonian diffeomorphisms
$\varphi$ with finitely many periodic orbits are strongly
non-degenerate, i.e., $\varphi^k$ is non-degenerate for all $k\in\N$.

Volume preserving diffeomorphisms or flows with finitely many simple
periodic orbits play an important role in dynamics; see, e.g.,
\cite{FK} and references therein. In the Hamiltonian setting they are
sometimes referred to as pseudo-rotations. Recently, symplectic
topological methods have been employed to study the dynamics of
pseudo-rotations and its connections with symplectic topological
properties of the underlying manifold in all dimensions; see \cite{AS,
  Ba, Br, Br:Annals, CGG:CMD, CGG:St, GG:PR, LRS, Sh, Sh:new}.

The original proof of Franks' theorem utilized methods from
low-dimensional dynamics, and the first purely symplectic topological
proof was given in \cite{CKRTZ}. However, that proof and also a
different approach from \cite{BH} were still strictly low-dimensional,
and Shelukhin's theorem, \cite[Thm.\ A]{Sh:HZ}, is the first
sufficiently general higher-dimensional variant of Franks'
theorem. (Strictly speaking, \cite[Thm.\ A]{Sh:HZ} and our Theorem
\ref{thm:Sh} and Corollary \ref{cor:Sh}, which are overall slightly
weaker, still fall short of completely reproving Franks' theorem in
dimension two; we will discuss and compare these results in Section
\ref{sec:results}.) Similarly to \cite{Sh:HZ}, the key ingredient of
our proof is Seidel's $\Z_2$-equivariant pair-of-pants product,
\cite{Se}. (While we use the original version of the product,
\cite{Sh:HZ} relies on its $\Z_p$-equivariant version from
\cite{ShZ}.) Our proof also uses several simple ingredients from
persistent homology theory in the form developed in \cite{UZ} (see
also \cite{PS}), although to a much lesser degree than \cite{Sh:HZ}.

Finally, it is worth pointing out that Hamiltonian pseudo-rotations
are extremely rare and most of the manifolds do not admit such
maps. This statement is known as the Conley conjecture. The state of
the art result is that the Conley conjecture holds for a manifold
$(M,\omega)$ unless there exists $A\in\pi_2(M)$ such that
$\left<c_1(TM),\,A\right>>0$ and $\left<\omega,\,A\right>>0$; see
\cite{Ci,GG:Rev}. For example, the Conley conjecture holds when
$c_1(TM)|_{\pi_2(M)}=0$ or when $M$ is negative monotone. For many
manifolds the conjecture is also known to hold $C^\infty$-generically
(see \cite{GG:generic}); we refer the reader to \cite{GG:survey} for a
detailed survey and further references.

\subsection{Shelukhin's theorem}
\label{sec:results}
Let $\varphi$ be a Hamiltonian diffeomorphism of a closed monotone
symplectic manifold $M$. We view $\varphi$ as the time-one map in a
time-dependent Hamiltonian flow and denote by $\PP_k(\varphi)$ the set
of its $k$-periodic points, arising from contractible $k$-periodic
orbits.  The Hamiltonian diffeomorphism $\varphi$ is said to be
\emph{$k$-perfect} if $\PP_k(\varphi)=\PP_1(\varphi)$ and
\emph{perfect} if $\varphi$ is $k$-perfect for all $k\in\N$. (We refer
the reader to Sections \ref{sec:prelim} for further notation and
definitions used here.) We call $\varphi$ a non-degenerate
\emph{pseudo-rotation} over a field $\F$ if it is non-degenerate,
perfect and the differential in the Floer complex of $\varphi$ over
$\F$ vanishes. This condition is independent of the choice of an
almost complex structure and, by Arnold's conjecture, equivalent to
that the number of 1-periodic orbits $|\PP_1(\varphi)|$ is equal to
the sum of Betti numbers of $M$ over $\F$. Denote by $\beta(\varphi)$
the \emph{barcode norm} of $\varphi$ over $\F$, i.e., the length of
the maximal finite bar in the \emph{barcode} of $\varphi$; see
Section~\ref{sec:barcodes}.

One of the goals of this paper is to give a simple proof to the
following theorem proved in a slightly different form in \cite{Sh:HZ}.

\begin{Theorem}[Shelukhin's Theorem, \cite{Sh:HZ}]
  \label{thm:Sh}
  Assume that $\varphi$ is strongly non-degenerate and perfect and
  that $\beta(\psi)$ over $\F_2:=\Z_2$ is bounded from above for all
  Hamiltonian diffeomorphisms $\psi$ of $M$ or at least for all
  iterates $\psi=\varphi^{2^k}$ (e.g., $M=\CP^n$).  Then $\varphi$ is
  a pseudo-rotation.
\end{Theorem}

Applying this to the iterates $\varphi^{2^k}$ we obtain

\begin{Corollary}[\cite{Sh:HZ}]
    \label{cor:Sh}
    Assume that $\varphi$ is strongly non-degenerate,
    $\beta\big(\varphi^{2^k}\big)$ over $\F_2$ is bounded from above
    (e.g., $M=\CP^n$), and $|\PP_1(\varphi)|$ is strictly greater than
    the sum of Betti numbers of $M$ over $\F_2$. Then
    $\big|\PP_{2^k}(\varphi)\big|\to\infty$ as $k\to\infty$.
  \end{Corollary}

  This theorem is proved in Section \ref{sec:implications} as an easy
  consequence of Theorem \ref{thm:main}, a new result in this
  paper. (However, at least on the conceptual level, our proof of that
  theorem is also a subset of Shelukhin's argument, although the
  inclusion is rather implicit.)
  
  In the rest of this section we discuss the conditions of Theorem
  \ref{thm:Sh} and also some of the differences between Corollary
  \ref{cor:Sh} and the original Shelukhin's theorem, \cite[Thm.\
  A]{Sh:HZ}, which is in several ways more general and more precise.

  First of all, the coefficient field in \cite[Thm.\ A]{Sh:HZ} is $\Q$
  rather than $\F_2$ and the assertion is that $\PP_p(\varphi)$
  contains a simple periodic orbit for every large prime $p$. As a
  consequence, one obtains the growth of order at least $O(k/\log k)$
  for the number of simple periodic orbits of period up to $k$.  This
  difference stems from the fact that the main tool used in
  \cite{Sh:HZ} is the $\Z_p$-equivariant pair-of-pants product
  introduced in \cite{ShZ} while we rely on a somewhat simpler
  $\Z_2$-equivariant pair-of-pants product defined in \cite{Se}. We
  touch upon the $p$-iterated analogues of Theorem \ref{thm:Sh} and
  Corollary \ref{cor:Sh} in Remark \ref{rmk:char-p2}.

  Secondly, \cite[Thm.\ A]{Sh:HZ} allows for some degeneracy of
  $\varphi$. Namely, in the setting of Corollary \ref{cor:Sh}, the
  number of 1-periodic orbits $|\PP_1(\varphi)|$ in the condition that
  $|\PP_1(\varphi)|$ is strictly greater than the sum of Betti numbers
  is replaced by
  \begin{equation}
    \label{eq:local-sum}
    \sum_{x\in\PP_1(\varphi)}\dim_{\F}\HF(x;\F),
  \end{equation}
  where $\HF(x;\F)$ is the local Floer (co)homology of $x$ with
  coefficients in a field $\F$ (see, e.g., \cite{GG:gap}); $\F=\Q$ in
  \cite{Sh:HZ}. Note that, as a consequence, Corollary \ref{cor:Sh}
  still holds without the non-degeneracy assumption, provided that the
  number of 1-periodic orbits with $\HF(x;\F)\neq 0$ is greater than
  the sum of Betti numbers. In the setting of this paper, one should
  take $\F=\F_2$ and we will further discuss the degenerate case of
  Theorem \ref{thm:Sh} and Corollary \ref{cor:Sh} in Section
  \ref{sec:degenerate}. Overall, the role of the condition that
  $\HF(x;\F)\neq 0$ is unclear to us beyond the case of $S^2$.
  Franks' theorem has an analogue for a certain class of
  symplectomorphisms of surfaces and then, interestingly, this
  condition becomes essential; see \cite{Bat, GG:generic}.

  However, from our perspective, the most important difference lies in
  the proofs, which highlight different aspects of the dynamics and
  Floer theory of $\varphi$. Our proof focuses on the behavior of the
  shortest bar $\beta_{\min}$ in the barcode of $\varphi$ (rather than
  the longest bar $\beta\geq \beta_{\min}$) or, to be more precise, of
  the shortest Floer arrow under the iteration from $\varphi$ to
  $\varphi^2$; see Section \ref{sec:main}. In particular, we show in
  Theorem \ref{thm:main} that when $\varphi$ is 2-perfect the shortest
  arrow persists under such an iteration, although it may migrate into
  the equivariant domain for $\varphi^2$, and the length of the arrow
  doubles.  The shortest non-equivariant arrow for $\varphi^2$ is at
  least as long as the equivariant one. Hence
  $\beta_{\min}\big(\varphi^2\big)\geq 2\beta_{\min}(\varphi)$, and
  Theorem \ref{thm:Sh} readily follows from Theorem \ref{thm:main}
  applied to a sequence of period doubling iterations; see Section
  \ref{sec:implications}. The key ingredient in the proof of Theorem
  \ref{thm:main} is the equivariant pair-of-pants product, introduced
  in \cite{Se}, having a very strong non-vanishing property also
  proved therein (see Proposition \ref{prop:non-van}).

  Finally, a few words are due on the requirement in Theorem
  \ref{thm:Sh} and Corollary \ref{cor:Sh} that $\beta(\psi)$ is
  bounded from above. First of all, note that while it would be
  sufficient to only have an upper bound on $\beta_{\min}(\psi)$ where
  $\psi=\varphi^{2^k}$ or, as in \cite[Thm.\ A]{Sh:HZ}, on
  $\beta(\psi)$ where $\psi=\varphi^{p}$, all relevant results proved
  to date are more robust and give an upper bound on $\beta(\psi)$ for
  all $\psi$. (This is the curse (and the blessing) of symplectic
  topological methods in dynamics: they are very robust and general,
  but not particularly discriminating; they often tell the same thing
  about all maps. There are, however, exceptions.)

  The simplest manifold for which such an \emph{a priori} bound is
  established is $\CP^n$ for any coefficient field (suppressed in the
  notation), and the result essentially goes back to \cite{EP}. The
  argument is roughly as follows. (We use here the notation and
  conventions from Section \ref{sec:conv}.) First recall that
\begin{equation}
\label{eq:beta-gamma}
\beta(\psi)\leq \gamma(\psi).
\end{equation}
Here $\gamma(\psi)$ is the $\gamma$-norm of $\psi$ defined, using
cohomology, as
\[
  \gamma(\psi)=-\big(\cf_{\unit}(\psi)+\cf_{\unit}(\psi^{-1})\big),
\]
where $\cf_\alpha(\psi)$ is the spectral invariant associated with a
quantum cohomology class $\alpha\in\HQ(M)$ and $\unit$ is the unit in
the ordinary cohomology $\H(M)$ of $M$. (We suppress the grading in
the cohomology notation when it is irrelevant.)  The upper bound
\eqref{eq:beta-gamma} holds for any closed monotone symplectic
manifold and its proof is similar to the proof in \cite{Us} of the
upper bound for $\beta$ by the Hofer norm, but with continuation maps
replaced by the multiplications by the image of $\unit$ in $\HF(\psi)$
and $\HF\big(\psi^{-1}\big)$. (We refer the reader to \cite{KS} for
some further results along these lines.) Applying the Poincar\'e
duality in Floer cohomology (see \cite{EP}), it is not hard to show
that $\cf_{\unit}(\psi^{-1})=-\cf_{\varpi}(\psi)$ when $N\geq n+1$,
where $\varpi$ is the generator of $\H^{2n}(M)$ and $N$ is the minimal
Chern number of $M^{2n}$. In particular, this is true for $M=\CP^n$
since then $N=n+1$. By construction, for any two classes $\alpha$ and
$\zeta$ in $\HQ(M)$ the spectral invariants satisfy the
Lusternik--Schnirelmann inequality
$\cf_{\alpha*\zeta}(\psi)\geq \cf_\alpha(\psi)$. Thus, from the
identity $\varpi * \zeta =\qq\unit$ where $\zeta$ is the generator of
$\HQ^2(\CP^n)$, we conclude that
$\cf_{\unit}(\psi)\leq\cf_{\varpi}(\psi)\leq\cf_{\unit}(\psi)+\pi$. These
inequalities, combined with \eqref{eq:beta-gamma}, show that
\[
  \beta(\psi)\leq \gamma(\psi)\leq \pi
\]
for any Hamiltonian diffeomorphism $\psi$ of $\CP^n$.

A similar upper bound on $\beta$ holds for all closed monotone
manifolds $M$ such that $\HQ^{\even}(M;\F)$ for some field $\F$ is
semi-simple, i.e., splits as an algebra into a direct sum of
fields. This is \cite[Thm.\ B]{Sh:HZ} and, interestingly, this result
bypasses the upper bound \eqref{eq:beta-gamma} in its original
form. In fact, $\HQ(S^2\times S^2;\Q)$ is semi-simple, but $\gamma$ is
not bounded from above for $S^2\times S^2$; see \cite[Rmk.\ 7]{Sh:HZ}
and also \cite[Thm.\ 6.2.6]{RP}. We are not aware of any algebraic
criteria for an \emph{a priori} bound on the $\gamma$-norm. Nor do we
know how large the class of monotone symplectic manifolds with
semi-simple $\HQ^{\even}(M;\F)$ is. In addition to $\CP^n$ (with any
$\F$), the complex Grassmannians, $S^2\times S^2$, and the one point
blow-up of $\CP^2$ with standard monotone symplectic structures are in
this class when $\charr\F=0$ (see \cite{EP} and references therein);
but $S^2\times S^2$ is not for $\F=\F_2$.

\medskip\noindent{\bf Acknowledgements.} The authors are grateful to
Egor Shelukhin for useful discussions.

\section{Preliminaries}
\label{sec:prelim}
\subsection{Conventions and notation}
\label{sec:conv}
For the reader's convenience we set here our conventions and notation
and briefly recall some basic definitions. The reader may want to
consult this section only as needed.

Throughout this paper, the underlying symplectic manifold $(M,\omega)$
is assumed to be closed and strictly monotone, i.e.,
$[\omega]|_{\pi_2 (M)}=\lambda c_1(TM)|_{\pi_2(M)}\neq 0$ for some
$\lambda>0$. The \emph{minimal Chern number} of $M$ is the positive
generator $N$ of the subgroup
$\left<c_1(TM),\pi_2(M)\right>\subset \Z$ and the \emph{rationality
  constant} is the positive generator $\lambda_0=2N\lambda$ of the
group $\left<\omega,\pi_2(M)\right>\subset \R$.

A \emph{Hamiltonian diffeomorphism} $\varphi=\varphi_H=\varphi_H^1$ is
the time-one map of the time-dependent flow $\varphi^t=\varphi_H^t$ of
a 1-periodic in time Hamiltonian $H\colon S^1\times M\to\R$, where
$S^1=\R/\Z$.  The Hamiltonian vector field $X_H$ of $H$ is defined by
$i_{X_H}\omega=-dH$. In what follows, it will be convenient to view
Hamiltonian diffeomorphisms together with the path $\varphi_H^t$,
$t\in [0,\,1]$, up to homotopy with fixed end points, i.e., as
elements of the universal covering of the group of Hamiltonian
diffeomorphisms.

Let $x\colon S^1\to M$ be a contractible loop. A \emph{capping} of $x$
is an equivalence class of maps $A\colon D^2\to M$ such that
$A|_{S^1}=x$. Two cappings of $x$ are equivalent if the integral of
$\omega$ (or of $c_1(TM)$ since $M$ is strictly monotone) over the
sphere obtained by clutching the cappings is equal to zero. A capped
closed curve $\bar{x}$ is, by definition, a closed curve $x$ equipped
with an equivalence class of cappings, and the presence of capping is
indicated by a bar.

The action of a Hamiltonian $H$ on a capped closed curve
$\bar{x}=(x,A)$ is
\[
  \CA(\bar{x})=-\int_A\omega+\int_{S^1} H_t(x(t))\,dt.
\]
The space of capped closed curves is a covering space of the space of
contractible loops, and the critical points of $\CA_H$ on this space
are exactly the capped 1-periodic orbits of $X_H$.

The $k$-periodic \emph{points} of $\varphi$ are in one-to-one
correspondence with the $k$-periodic \emph{orbits} of $H$, i.e., of
the time-dependent flow $\varphi^t$. Recall also that a $k$-periodic
orbit of $H$ is called \emph{simple} if it is not iterated.  A
$k$-periodic orbit $x$ of $H$ is said to be \emph{non-degenerate} if
the linearized return map $D\varphi^k \colon T_{x(0)}M \to T_{x(0)}M$
has no eigenvalues equal to one.  A Hamiltonian $H$ is non-degenerate
if all its 1-periodic orbits are non-degenerate. We denote the
collection of capped $k$-periodic orbits of $H$ by $\bPP_k(\varphi)$.

Let $\bar{x}$ be a non-degenerate capped periodic orbit.  The
\emph{Conley--Zehnder index} $\mu(\bar{x})\in\Z$ is defined, up to a
sign, as in \cite{Sa,SZ}. In this paper, we normalize $\mu$ so that
$\mu(\bar{x})=n$ when $x$ is a non-degenerate maximum (with trivial
capping) of an autonomous Hamiltonian with small Hessian.

Fixing an almost complex structure, which will be suppressed in the
notation, we denote by $\left(\CF(\varphi), d_{\Fl}\right)$ and
$\HF(\varphi)$ the Floer complex and cohomology of $\varphi$ over
$\F_2=\Z_2$; see, e.g., \cite{MS, Sa}. (Throughout this paper, all
complexes and cohomology groups are over $\F_2$.) The complex
$\CF(\varphi)$ is generated by the capped 1-periodic orbits $\bx$ of
$H$, graded by the Conley--Zehnder index, and filtered by the
action. The filtration level (or the action) of a chain
$\xi\in \CF(\varphi)$ is defined by
\begin{equation}
  \label{eq:filt}
\CA(\xi)=\min \{\CA(\bx_i)\},\textrm{ where } \xi=\sum\bx_i.
\end{equation}
(Note that the filtration depends on $H$, not just on $\varphi$,
making of the notation $\CF(\varphi)$ somewhat misleading.) The
differential $d_{\Fl}$ is the upward Floer differential: it increases
the action and also the index by one. The Floer complex $\CF(\varphi)$
is also a finite-dimensional free module over the Novikov ring
$\Lambda$. There are several choices of $\Lambda$; see, e.g.,
\cite{MS}. For our purposes, it is convenient to take the field of
Laurent series $\F_2((\qq))$ with $|\qq|=2N$ as $\Lambda$. With this
choice, $\Lambda$ naturally acts on $\CF(\varphi)$ by recapping, and
multiplication by $\qq$ corresponds to the recapping by $A\in\pi_2(M)$
with $\left<c_1(TM),A\right>=N$. Furthermore, $\CF(\varphi)$ is a
finite-dimensional vector space over $\Lambda$ with a preferred basis
formed by 1-periodic orbits with arbitrarily fixed capping.

Notationally, it is convenient to equip $\CF(\varphi)$ with a
non-degenerate $\F_2$-valued pairing $\left<\,,\right>$ for which
$\bPP_1(\varphi)$ is an orthogonal basis:
$\left<\bx,\,\by\right>=\delta_{\bx\by}$. Then, essentially by
definition,
\[
d_{\Fl}\bx=\sum \left<d_{\Fl}\bx,\,\by\right>\by.
\]

There is a canonical, grading-preserving isomorphism
$ \HF(\varphi) \stackrel{\cong}{\longrightarrow} \HQ(M)[-n]$ where
$\HQ(M)$ is the quantum cohomology of $M$; see, e.g., \cite{Sa,MS} and
references therein. (Depending on the context, this is the
PSS-isomorphism or the continuation map or a combination of the two.)
The cohomology groups $\HQ(M)$ and $\HF(\varphi)$ are also modules
over a Novikov ring $\Lambda$, and
$\HQ(M)\cong \H(M)\otimes \Lambda\cong \HF(\varphi)$ (as a module).

The Floer complex carries a pairing
\[
  \CF(\varphi)\otimes \CF(\varphi)\to \CF\big(\varphi^2\big)[n]
\]
descending, on the level of cohomology, to the so-called
\emph{pair-of-pants product}
\[
\HF(\varphi)\otimes \HF(\varphi)\to \HF\big(\varphi^2\big)[n],
\]
which we denote by $*$. Thus with our conventions
$|\alpha * \beta|=|\alpha|+|\beta|+n$. In quantum cohomology, this
product corresponds to the \emph{quantum product}, also denoted by
$*$, which makes it into a graded-commutative algebra over $\Lambda$
with unit $\unit$. This product is a deformation (in $\qq$) of the cup
product: $\alpha*\beta=\alpha\cup\beta+O(\qq)$.

\subsection{Equivariant Floer cohomology and the pair-of-pants
  product}
\label{sec:eq-Floer homology}

\subsubsection{Equivariant Floer cohomology: a brief introduction}
The \emph{equivariant Floer cohomology}
$\HF_{\eq}\big(\varphi^2\big)$, introduced in \cite{Se}, is the
homology of a certain complex
$\big(\CF_{\eq}\big(\varphi^2\big),\,d_{\eq}\big)$ called the
\emph{equivariant Floer complex}. As a graded $\F_2$-vector space or
as a $\Lambda$-module,
\[
  \CF_{\eq}\big(\varphi^2\big)=\CF\big(\varphi^2\big)[\hh]
\]
where $|\hh|=1$, and the differential $d_{\eq}$ has the form
\[
  d_{\eq}=d_{\Fl} + \hh\, d_1 + \hh^2\, d_2+\ldots =d_{\Fl}+ O(\hh).
\]
This differential is $\Lambda[h]$-linear and non-strictly
action-increasing. It is roughly speaking defined as follows,
mimicking Borel's construction of the $\Z_2$-equivariant Morse
cohomology.

Fix a family $\tJ$ of 2-periodic in $t$ almost complex structures on
$M$ parametrized by the unit infinite-dimensional sphere
$S^\infty\subset \R^{\infty}$. Here $\R^\infty$ is the direct sum of
infinitely many copies of $\R$, i.e., its elements
$\xi= (\xi_0,\,\xi_1,\,\ldots)$ have only finitely many non-zero
components, and $S^{\infty}=\{\|\xi\|=1\}$ with
$\|\xi\|^2=\sum_k|\xi_k|^2$.  The almost complex structure $\tJ$ is
required to satisfy the symmetry condition $\tJ_{-\xi}=\tJ'_\xi$,
where $\tJ'_{\xi}$ is obtained from $\tJ_\xi$ by the time-shift
$t\mapsto t+1$. Consider the self-indexing quadratic form
$f(\xi)=\sum_kk|\xi_k|^2$ on $S^{\infty}$ and an antipodally symmetric
metric such that the natural equatorial embedding
$S^\infty\to S^\infty$ given by
$(\xi_0,\,\xi_1,\,\ldots)\mapsto (0,\,\xi_0,\,\ldots)$ is an
isometry. (Note also that the pull back of $f$ by this embedding is
$f+1$.)  The almost complex structure $\tJ$ must furthermore be
constant in $\xi$ near the critical points of $f$, invariant under the
equatorial embedding, and satisfy a certain regularity
requirement. Denote by $w^\pm_k$ the critical points of $f$ of index
$k$.

Next, consider the hybrid Morse-Floer complex of $\CA+f$ with respect
to $\tJ$ and the metric on $S^\infty$. This complex has pairs
$(\bx,w^\pm_k)$ with $\bx\in \bPP_2(\varphi)$ as generators and
carries a natural $\Z_2$-action, free on the generators, sending
$(\bx,w^\pm_k)$ to $(\bx',w^{-\pm}_k)$, where $\bx'$ is the time-shift
of $\bx$. It is easy to see that the homology of this hybrid complex
is equal to $\HF (\varphi^2)$. By definition,
$\CF_{\eq}\big(\varphi^2\big)$ is the $\Z_2$-invariant part of this
hybrid complex, where we write $\bx\,\hh^k$ for
$(\bx,w^+_k)+(\bx',w^-_k)$. The fact that the differential is
$\hh$-linear follows from the requirement that $f$ (up to a constant)
and the auxiliary data are invariant under the equatorial embedding.
Thus, in self-explanatory notation,
\[
  d_k\bx=\sum \left<d_k\bx,\,\hh^k\by\right>\by, \textrm{ where
  }\mu(\by)=\mu(\bx)+1-k
\]
and $\left<d_k\bx,\,\hh^k\by\right>$ counts mod 2 the total number of
continuation Floer trajectories from $\bx$ to $\by$ along gradient
lines of $f$ connecting $w^+_0$ to $w^+_k$ and from $\bx$ to $\by'$
along gradient lines of $f$ connecting $w^+_0$ to $w^-_k$. Clearly,
the complex (and hence its cohomology) is filtered by the action $\CA$
in addition to the filtration by $\CA+f$. On the level of (co)chains
the filtration is defined similarly to \eqref{eq:filt}, but with the
powers of $\hh$ ignored:
\[
  \CA(\xi)=\min \{\CA(\bx_i)\},\textrm{ where }
  \xi=\sum\hh^{m_i}\bx_i.
\]
The equivariant complex and the cohomology has natural continuation
properties; see \cite{Se}.

\begin{Example}
  \label{ex:C2small}
  Assume that $\varphi$ is 2-perfect and $\varphi^2$ admits a regular
  1-periodic almost complex structure $J$, i.e., for every pair $\bx$
  and $\by$ of 2-periodic orbits the space of Floer trajectories
  connecting $\bx$ to $\by$ has dimension $\mu(\by)-\mu(\bx)$. In
  particular, this space is empty when $\mu(\by)\leq \mu(\bx)$, except
  when $\by=\bx$ and the space comprises one constant trajectory. Set
  $\tJ=J$ to be a constant (i.e., independent of $\xi$) almost complex
  structure. Then $\tJ$ is also regular and $d_j=0$ for $j\geq 1$
  since continuation trajectories for a constant homotopy are just
  Floer trajectories. Thus, in this case,
  $\HF_{\eq}\big(\varphi^2\big)=\HF(\varphi)[\hh]$ for any interval of
  action. These conditions are met, for instance, when
  $\varphi=\varphi_H$ is generated by a $C^2$-small autonomous
  Hamiltonian $H$. As a consequence, for any $\varphi$ the global
  cohomology $\HF_{\eq}\big(\varphi^2\big)$ is not a particularly
  interesting object: it is simply isomorphic to $\HQ(M)[\hh]$ via the
  equivariant continuation (or the PSS map); see \cite{Wi1,Wi2} for
  further details.
\end{Example}

\begin{Remark}[Polynomials vs.\ Formal Power Series]
  \label{rmk:pol}
  One difference between our definition of
  $\CF_{\eq}\big(\varphi^2\big)$ and the one in \cite{Se} is that
  there $\CF_{\eq}\big(\varphi^2\big)=\CF\big(\varphi^2\big)[[\hh]]$;
  for in that setting the expansion $d_{\eq}\bx=\sum_k\hh^kd_k\bx$ may
  have infinitely many non-vanishing terms. However, as already
  pointed out in \cite[Sect.\ 7]{Se}, when $M$ is strictly monotone
  this expansion is necessarily finite. Indeed, otherwise it would
  involve capped orbits $\by\in\bPP_2(\varphi)$ with arbitrarily small
  index $\mu(\by)$. However, due to monotonicity and since
  $\PP_2(\varphi)$ is finite, such orbits would eventually have action
  strictly smaller than that of $\bx$, which is impossible. This
  difference is essential for our proof as at some point in the
  argument we evaluate the elements of $\CF_{\eq}\big(\varphi^2\big)$
  at $\hh=1$.
\end{Remark}

\subsubsection{Equivariant pair-of-pants product}
For our purposes, the most important feature of the equivariant Floer
complex is that it is the target space of the \emph{equivariant
  pair-of-pants product}, also defined in \cite{Se}. On the level of
complexes this product is a chain map
\[
  \wp\colon \Co\big(\Z_2;\CF(\varphi)\otimes\CF(\varphi)\big)\to
  \CF_{\eq}\big(\varphi^2\big).
\]
The domain of $\wp$ is the group cochain complex
\[
  \Co\big(\Z_2;\CF(\varphi)\otimes\CF(\varphi)\big)
  :=\CF(\varphi)\otimes\CF(\varphi)[\hh]
\]
with the differential
\[
  d_{\Z_2}=d_{\Fl}+\hh(\id+\tau).
\]
Here $\tau$ is the involution $\tau(\bx\otimes\by)=\by\otimes \bx$ and
the first term is induced by the Floer differential on
$\CF(\varphi)\otimes\CF(\varphi)$. Note also that in these formulas
and throughout the paper, all tensor products are over $\F_2$ unless
specified otherwise. Furthermore, we distinguish between $\F_2$ and
$\Z_2$: the former is a field and the latter is a group.

The equivariant pair-of-pants product is bilinear over $\Lambda[\hh]$
and respects the action filtration. In particular, it can also be
defined for a fixed action interval $[a,\,b]$ in the domain and
$[2a,\,2b]$ in the target, but here we will not need the filtered
version of this construction. The map $\wp$ is a perturbation of the
ordinary pair-of-pants product:
\begin{equation}
  \label{eq:wp-popp}
\wp(\bx\otimes\by)=\bx * \by + O(\hh),
\end{equation}
and the $O(\hh)$ part is again polynomial in $\hh$ involving only
finitely many terms (depending on $\bx$ and $\by$).

The cohomology of the domain of $\wp$ is the group cohomology
$\H\big(\Z_2;\CF(\varphi)\otimes\CF(\varphi)\big)$ of $\Z_2$ with
coefficients in $\CF(\varphi)\otimes\CF(\varphi)$.  Thus, on the level
of cohomology, the equivariant pair-of-pants product turns into a
homomorphism
\begin{equation}
  \label{eq:pop-eq-coh}
  \H\big(\Z_2;\CF(\varphi)\otimes\CF(\varphi)\big)\cong
  \H\big(\Z_2;\HF(\varphi)\otimes\HF(\varphi)\big)\to
  \HF_{\eq}\big(\varphi^2\big).
\end{equation}
(The first isomorphism is a consequence of the fact that
$\CF(\varphi)\otimes\CF(\varphi)$ and
$\HF(\varphi)\otimes\HF(\varphi)$ are equivariantly quasi-isomorphic.)
The map \eqref{eq:pop-eq-coh} obviously kills the $\hh$-torsion in the
domain; it is a deformation in $\hh$ of the standard pair-of-pants
product due to \eqref{eq:wp-popp} and is closely related to a quantum
deformation of the Steenrod squares; see \cite{Se,Wi1,Wi2} and also
\cite{CGG:St} for a short introduction. The map \eqref{eq:pop-eq-coh}
is a monomorphism modulo $\hh$-torsion; \cite{Sh:HZ}.  For
symplectically aspherical manifolds, but not in the strictly monotone
case, \eqref{eq:pop-eq-coh} is also onto and hence an isomorphism
modulo $\hh$-torsion i.e., the kernel and the cokernel are torsion
modules; see \cite{Se}.

On the level of complexes $\wp$ has the following extremely important
feature:

\begin{Proposition}[Seidel's non-vanishing theorem; \cite{Se}, Prop.\
  6.7]
  \label{prop:non-van}
  For every $\bx\in\bPP_1(\varphi)$, we have
  \begin{equation}
    \label{eq:non-van}
  \wp(\bx\otimes\bx)=\hh^m\bx^2+\ldots,
\end{equation}
where $\bx^2\in \bPP_2(\varphi)$ is the second iterate of $\bx$ and
$m=2\mu(\bx)-\mu\big(\bx^2\big)+n$ and the dots stand for a sum of
capped orbits with action strictly greater than $2\CA(\bx)$.
\end{Proposition}

This non-vanishing property points to a stark difference between the
equivariant and non-equivariant pair-of-pants products:
$\bx * \bx=\bx^2+\ldots$ only when $\mu\big(\bx^2\big) =2\mu(\bx)+n$,
i.e., $m=0$ in \eqref{eq:non-van}; cf.\ \cite{CGG:CMD}.

\begin{Remark}
  \label{rmk:char-p}
  A generalization of the equivariant pair-of-pants product to the
  $p$-th iterates $\varphi^p$, where $p$ is a prime, replacing $\Z_2$
  by $\Z_p$ and $\F_2$ by $\F_p$ is constructed in \cite{ShZ}. This
  construction and the analogue of Seidel's non-vanishing theorem for
  the $p$-th iterate plays a crucial role in the original proof of
  Shelukhin's theorem in \cite{Sh:HZ}; cf. Remark \ref{rmk:char-p2}.
\end{Remark}

\section{Floer graphs}
\label{sec:Floer_Gr}

\subsection{Main result}
\label{sec:main}
The key to the statement of our main result is the following
admittedly naive and obvious construction which has been used, at
least on an informal level, for quite some time.

Let $\varphi$ be a non-degenerate Hamiltonian diffeomorphism of a
closed monotone symplectic manifold $M$. Consider the directed graph
$\Gamma(\varphi)$ whose vertices are capped fixed points of $\varphi$,
and two vertices $\bx$ and $\by$ are connected by an arrow (from $\bx$
to $\by$) if and only if $\mu(\by)=\mu(\bx)+1$ and there is an odd
number of Floer trajectories from $\bx$ to $\by$, i.e.,
$\left<d_{\Fl}\bx,\,\by\right>=1$. The length of an arrow is the
difference of actions of $\by$ and $\bx$. We call $\Gamma(\varphi)$
the \emph{Floer graph} of $\varphi$.

When $M$ is strictly monotone as is always assumed in this paper, the
group $\Z$ acts freely on $\Gamma(\varphi)$ by simultaneous recapping,
preserving the arrow length. Sometimes it is convenient to consider
the \emph{reduced Floer graph}
$\tGamma(\varphi):=\Gamma(\varphi)/\Z$. The length of an arrow in
$\tGamma(\varphi)$ is still well-defined. Note that, unless $M$ is
symplectically aspherical, both $\Gamma(\varphi)$ and
$\tGamma(\varphi)$ are infinite, but the latter has finitely many
arrows. In particular, if $d_{\Fl}\neq 0$, there exists a shortest
arrow. Such an arrow might not be unique, although it is unique for a
generic $\varphi$, but obviously all shortest arrows have the same
length.

The \emph{equivariant Floer graph} $\Gamma_{\eq}\big(\varphi^2\big)$
of $\varphi^2$ is defined in a similar fashion. (We are assuming that
$\varphi^2$ is non-degenerate, and hence $\varphi$ is also
non-degenerate.) Its vertices are capped two-periodic orbits of
$\varphi$. The vertices $\bx$ and $\by$ are connected by an arrow if
and only if $\by$ enters $d_{\eq}(\bx)$ with non-zero coefficient. In
other words, now we do not require the index difference to be 1, and
$\bx$ and $\by$ are connected by an arrow if and only if $\bx$ and
$\hh^m\by$, where $m= \mu(\bx)-\mu(\by)+1$, are connected by an odd
number of equivariant Floer trajectories. The length of an arrow is
again the difference of actions. As in the non-equivariant case, the
\emph{reduced equivariant Floer graph}
$\tGamma_{\eq}\big(\varphi^2\big)
:=\tGamma_{\eq}\big(\varphi^2\big)/\Z$ has only finitely many arrows,
and hence the shortest arrows exist.

We note that $\Gamma\big(\varphi^2\big)$ and
$\Gamma_{\eq}\big(\varphi^2\big)$ (and their reduced counterparts)
have the same vertices.  Furthermore, since $d_{\eq}=d_{\Fl}+O(\hh)$,
every arrow in $\Gamma\big(\varphi^2\big)$ is also an arrow in
$\Gamma_{\eq}\big(\varphi^2\big)$, i.e., the equivariant Floer graph
is obtained from its non-equivariant counterpart by adding
arrows. Note that in the process the shortest arrow length can only
get shorter or remain the same. Also, observe that there is a natural
one-to-one map from the vertices of $\tGamma(\varphi)$ to the vertices
of $\tGamma\big(\varphi^2\big)$ sending $\bx$ to $\bx^2$; likewise for
un-reduced graphs. However, even when $\varphi$ is 2-perfect, this map
is not onto unless $M$ is symplectically aspherical.

The main new result of the paper is the following theorem which
relates the Floer graphs for $\varphi$ and its second iterate
$\varphi^2$.

\begin{Theorem}
  \label{thm:main}
  Assume that $\varphi$ is 2-perfect and $\varphi^2$ is
  non-degenerate. Then $\bx$ and $\by$ are connected by one of the
  shortest arrows in $\Gamma(\varphi)$ if and only $\bx^2$ and $\by^2$
  are connected by one of the shortest arrows in
  $\Gamma_{\eq}\big(\varphi^2\big)$.
\end{Theorem}

This theorem is proved in Section \ref{sec:proof} after we recall in
Section \ref{sec:barcodes} a few relevant facts about barcodes.

\begin{Remark}[The role of an almost complex structure]
  The Floer graph of $\varphi$ depends on the choice of an almost
  complex structure $J$, and hence should rather be denoted by
  $\Gamma(\varphi, J)$. Likewise, the equivariant Floer graph depends
  on the parametrized almost complex structure. However, in both
  cases, the collection of shortest arrows is independent of this
  choice. This fact implicitly follows from Theorem \ref{thm:main} or
  can be proved directly by a continuation argument.
\end{Remark}

Note also that Floer graphs are stable under small perturbations of
$\varphi$ and $J$. To be more precise,
$\Gamma(\varphi,J)=\Gamma(\tvarphi,\tJ$) whenever $\tvarphi$ is
sufficiently close to $\varphi$ and $\tJ$ is close to $J$. The same is
true in the equivariant setting.

\subsection{Implications and the proof of Theorem \ref{thm:Sh}}
\label{sec:implications}
Theorem \ref{thm:main} shows that when $\varphi$ is perfect, the
shortest arrow (or, to be more precise, every shortest arrow) persists
from $\varphi$ to $\varphi^2$, although in the process it might move
to the equivariant domain. This happens exactly when the difference of
indices changes: $\mu(\by)-\mu(\bx)=1$ but
$\mu\big(\by^2\big)-\mu\big(\bx^2\big)\neq 1$. Moreover, in this case,
we necessarily have $\mu\big(\by^2\big)-\mu\big(\bx^2\big) < 1$. On
the other hand, if the difference of indices remains equal to one, the
orbits continue to be connected by one of the shortest non-equivariant
arrows.

Denote by $\beta_{\min}(\varphi)=\CA(\by)-\CA(\bx)$ the length of a
shortest arrow.  As follows from Proposition \ref{prop:min_bar},
$\beta_{\min}(\varphi)$ is exactly equal to the shortest bar in the
barcode of $\varphi$. Since every non-equivariant arrow for
$\varphi^2$ is also an equivariant arrow, the shortest equivariant
arrow length $\beta^{\eq}_{\min}\big(\varphi^2\big)$ for $\varphi^2$
does not exceed $\beta_{\min}\big(\varphi^2\big)$, i.e.,
\[
  \beta^{\eq}_{\min}\big(\varphi^2\big)\leq
  \beta_{\min}\big(\varphi^2\big).
\]
In the setting of Theorem \ref{thm:main},
\[
  \beta^{\eq}_{\min}\big(\varphi^2\big)
  =\CA\big(\by^2\big)-\CA\big(\bx^2\big) = 2\beta_{\min}(\varphi).
\]
We conclude that 
\[
  2\beta_{\min}\big(\varphi^{2^k}\big)\leq
  \beta_{\min}\big(\varphi^{2^{k+1}}\big)
\]
as long as the iterates of $\varphi$ remain perfect and
non-degenerate, and hence
\[
  2^k\beta_{\min}(\varphi)\leq \beta_{\min}\big(\varphi^{2^k}\big).
\]
In particular, when $\varphi$ is perfect, the longest finite bar
$\beta(\varphi)$ (and even the shortest bar) in the barcode cannot be
bounded from above for the iterates of $\varphi$. This proves
Theorem~\ref{thm:Sh}.

\begin{Remark}
  An interesting question that arises from Theorem \ref{thm:main} is
  if a shortest arrow could persist in the non-equivariant domain for
  all iterates $\varphi^{2^k}$, assuming that $\varphi$ is perfect. As
  discussed above, this would be the case if and only if
  $\mu\big(\by^{2^k}\big)-\mu\big(\bx^{2^k}\big)=1$ for all $k\in
  \N$. Using a slightly simplified version of the index divisibility
  theorem from \cite{GG:PRvR} one can show that this is impossible
  when $\varphi$ is replaced by a suitable iterate $\varphi^m$. (This
  is non-obvious.) Passing to an iterate is apparently essential
  because there exist pairs of strongly non-degenerate elements $A$
  and $B$ in $\tSp(2n)$ such that
  $\mu\big(A^{2^k}\big)-\mu\big(B^{2^k}\big)=1$ for
  all~$k=0,\,1,\,2,\,\ldots$.
  \end{Remark}

\section{A few words about the shortest bar}
\label{sec:barcodes}

In this section we recall a few facts, well-known to experts, about
persistent homology in the context of Hamiltonian Floer theory. All
results discussed here are contained in, e.g., \cite{UZ}, although in
some instances implicitly and usually in a much more general
setting. A reader sufficiently familiar with the material can easily
skip this section. There are, however, two points the reader might
want to keep in mind. Namely, our emphasis here is on the shortest bar
rather than the longest finite bar (aka the boundary depth) which is
more frequently used in applications to dynamics. Secondly, our sign
conventions are different from those in \cite{UZ} due to the fact that
we are working with Floer cohomology.

Consider the Floer complex $\CC:=\CF(\varphi)$ of a non-degenerate
Hamiltonian diffeomorphism $\varphi$ of a strictly monotone symplectic
manifold, equipped with the standard action filtration. Clearly, $\CC$
is a finite-dimensional vector space over $\Lambda$ and the collection
of 1-periodic orbits of $\varphi$ with fixed capping forms a basis of
$\CC$.

A finite set of vectors $\xi_i\in \CC$ is said to be \emph{orthogonal}
if for any collection of coefficients $\lambda_i\in\Lambda$ we have
\[
  \CA\big(\sum\lambda_i\xi_i\big)=\min \CA(\lambda_i\xi_i).
\]
(Recall that with our conventions, 
\[
  \CA(\xi):=\min \CA(\bx_i)\textrm{ when } \xi=\sum \bx_i;
\]
see \eqref{eq:filt}.) It is not hard to show that an orthogonal set is
necessarily linearly independent over $\Lambda$.

\begin{Example}
  \label{exam:orthogonality}
  Assume that all capped 1-periodic orbits of $\varphi$ have distinct
  actions. Write $\xi_i=\bx_i+\ldots$, where the dots stand for the
  orbits with action strictly greater than $\bx_i$. Then it is easy to
  see that the set $\xi_i$ is orthogonal if and only if the capped
  orbits $\bx_i$ are distinct.
\end{Example}

\begin{Definition}
  \label{def:sing_decomp}
  A basis $\CB=\{\alpha_i, \,\eta_j, \,\gamma_j\}$ of $\CC$ over
  $\Lambda$ is said to be a \emph{singular decomposition} if
\begin{itemize}
\item $d_{\Fl}\alpha_i=0$,
\item $d_{\Fl}\eta_j=\gamma_j$,
\item $\CB$ is orthogonal.
\end{itemize}
\end{Definition}

It is shown in \cite[Sections 2 and 3]{UZ} that $\CC$ admits a
singular decomposition. For the sake of brevity we omit the proof of
this fact. In what follows we will order the pairs
$(\eta_j,\,\gamma_j)$ so that
\begin{equation}
\label{eq:increasing_seq0}
\CA(\gamma_1)-\CA(\eta_1)\leq\CA(\gamma_2)-\CA(\eta_2)\leq\ldots .
\end{equation}
This increasing sequence is usually referred to as the \emph{barcode}
of $\varphi$ (or to be more precise the collection of finite bars).
The maximal entry in the sequence is called the \emph{barcode norm}
$\beta(\varphi)$ or the boundary depth, \cite{Us}. The barcode is
independent of the choice of a singular decomposition (see, e.g.,
\cite{UZ}), but here we do not use this fact. Instead, we need the
following characterization of the shortest bar
$\beta_{\min}=\beta_{\min}(\varphi)$:

\begin{Proposition}[\cite{UZ}]
  \label{prop:min_bar}
  Set
  \[
  \beta_{\min}:=\CA(\gamma_1)-\CA(\eta_1).
  \]
  Then
  \begin{align}
    \label{eq:beta1}
  \beta_{\min} &=\inf \left\{\CA(\by)-\CA(\bx)\mid
                 \left<d_{\Fl}\bx,\,\by\right>=1\right\}\\
    \label{eq:beta2}
               &=\inf \left\{\CA(d_{\Fl}\xi)-
                 \CA(\xi)\mid \xi\in\CC,\, \xi\neq 0\right\}.
  \end{align}
  Here, in the first equality, the infimum is taken over all capped
  1-periodic orbits $\bx$ and $\by$ such that $\by$ enters
  $d_{\Fl}\bx$ with non-zero coefficient and, in the second, over all
  non-zero $\xi\in\CC$. In particular, $\beta_{\min}(\varphi)$ is the
  shortest arrow length in $\Gamma(\varphi)$.
\end{Proposition}

Note that the infimums in \eqref{eq:beta1} and \eqref{eq:beta2} are
actually attained and thus can be replaced by minima, and that the
proposition can be thought of as an analogue for $\CC$ of the
Courant-Fischer minimax theorem giving a variational interpretation of
the eigenvalues of a quadratic form. For the sake of completeness we
include a proof of Proposition \ref{prop:min_bar}.

\begin{proof} Let us denote the right-hand sides in \eqref{eq:beta1}
  and \eqref{eq:beta1} by $\beta'_{\min}$ and, respectively,
  $\beta''_{\min}$. We claim that
  $\beta'_{\min}=\beta''_{\min}$. Indeed, setting $\xi=\bx$, in
  \eqref{eq:beta2}, it is easy to see that
  $\beta''_{\min}\leq \beta'_{\min}$. On the other hand, writing
  $\xi=\bx_1+\bx_2+\ldots$ in the order of increasing action and
  $d_{\Fl}\xi=\sum d_{\Fl}\bx_i=\by+\ldots$, we observe that
  $\left<\by,\,d_{\Fl}\bx_i\right>=1$ for some $i$. Then
  \begin{align*}
  \CA(d_{\Fl}\xi)-\CA(\xi) &= \CA(\by)-\CA(\bx_1)\\
                           &\geq \CA(\by)-\CA(\bx_i)\\
                           &\geq \beta'_{\min},
  \end{align*}
  and thus $\beta''_{\min}\geq \beta'_{\min}$.

  Next, clearly, $\beta_{\min}\geq \beta''_{\min}$. Therefore, it
  remains to show that $\beta_{\min}\leq \beta''_{\min}$. To this end,
  let us decompose $\xi$ in the basis $\CB$ over $\Lambda$:
  \[
    \xi=\sum
    \lambda_j\eta_j+\sum\lambda'_j\gamma_j+\sum\lambda''_i\alpha_i.
  \]
  Then
  \[
    d_{\Fl}\xi=\sum\lambda_j\gamma_j.
  \]
  By orthogonality,
  \[
    \CA(d_{\Fl}\xi)=\min\CA(\lambda_j\gamma_j)=\CA(\lambda_k\gamma_k)
  \]
  for some $k$, and, again by orthogonality,
  \[
    \CA(\xi) \leq \min \CA(\lambda_j\eta_j) \leq \CA(\lambda_k\eta_k).
  \]
  Therefore,
  \begin{align*}
    \CA(d_{\Fl}\xi)-\CA(\xi)
    &\geq \CA(\lambda_k\gamma_k)-\CA(\lambda_k\eta_k)\\
    &= \CA(\gamma_k)-\CA(\eta_k)\\
    &\geq \CA(\gamma_1)-\CA(\eta_1)=\beta_{\min}.
  \end{align*}
  As a consequence, $\beta_{\min}\leq \beta''_{\min}$, which finishes
  the proof of the proposition.
\end{proof}

\begin{Remark}
  \label{rmk:algebra}
  In conclusion, we point out that all results in this section are
  purely algebraic and extend in a straightforward way to any
  un-graded finite-dimensional complex over $\Lambda$ with an ``action
  filtration'' having expected properties; see \cite{UZ}.
\end{Remark}

\section{Proof of theorem \ref{thm:main} and further remarks}

\subsection{Proof of theorem \ref{thm:main}}
\label{sec:proof}
We begin by proving the theorem under the additional background
assumption that \emph{all actions and action differences for $\varphi$
  and $\varphi^2$ are distinct modulo the rationality constant
  $\lambda_0$}. Then, in the last step of the proof, we will show how
to remove this extra assumption. Note that in particular this
assumption guarantees that the shortest arrow is unique for
$\Gamma(\varphi)$ and $\Gamma_{\eq}\big(\varphi^2\big)$.

\begin{Remark} It is worth pointing out that while this background
  assumption is satisfied $C^\infty$-generically, it is not quite
  innocuous in the context of pseudo-rotations or perfect Hamiltonian
  diffeomorphisms. Indeed, in this case one can expect certain
  ``resonance relations'' between actions or actions and mean indices
  to hold; see \cite{GK,GG:generic}.
\end{Remark}

The proof is carried out in three steps.

\medskip\noindent\emph{Step 1: The shortest arrow for $\varphi$.}  In
this step we simply apply the machinery from Section
\ref{sec:barcodes} to $\CF(\varphi)$.  Let
$\CB=\{\alpha_i, \,\eta_j, \,\gamma_j\}$ be a singular decomposition
for $\CF(\varphi)$ over $\Lambda$; see Definition
  \ref{def:sing_decomp}.  Due to the background assumption, the
inequalities in \eqref{eq:increasing_seq0} are strict:
\begin{equation}
\label{eq:increasing_seq}
\CA(\gamma_1)-\CA(\eta_1)<\CA(\gamma_2)-\CA(\eta_2)<\ldots .
\end{equation}

Let us write
\[
  \gamma_1=\by_*+\ldots \textrm{ and } \eta_1=\bx_*+\ldots ,
\]
where dots stand for higher action terms, and $\bx_*$ and $\by_*$ are
unique by the background assumption. Then, by definition,
\[
  \CA(\gamma_1)=\CA(\by_*) \textrm{ and } \CA(\eta_1)=\CA(\bx_*),
\]
and hence
\[
  \beta_{\min}:=\CA(\gamma_1)-\CA(\eta_1)=\CA(\by_*)-\CA(\bx_*).
\]
We claim that
\begin{equation}
  \label{eq:x*y*}
  \left<d_{\Fl}\bx_*,\by_*\right>=1.
\end{equation}
Indeed, $\left<\,d_{\Fl}\bx, \by_*\right>=1$ for some $\bx$ entering
$\eta_1$. Then
\[
  \beta_{\min}=\CA(\by_*)-\CA(\bx_*)\geq \CA(\by_*)-\CA(\bx)\geq
  \beta_{\min}.
\]
It follows that the first inequality is in fact an equality and
$\bx=\bx_*$ due to the background assumption.

Therefore, by Proposition \ref{prop:min_bar} and \eqref{eq:x*y*},
$\bx_*$ and $\by_*$ are connected by the shortest arrow in
$\Gamma(\varphi)$.

\medskip\noindent\emph{Step 2: The shortest arrow for $\varphi^2$.} In
the previous step we have shown that $\bx_*$ and $\by_*$ are connected
by the shortest arrow in $\CF(\varphi)$. Our goal now is to prove the
following key fact.

\begin{Lemma}
  \label{lemma:key}
  The iterated orbits $\bx_*^2$ and $\by_*^2$ are connected by the
  shortest arrow in $\Gamma_{\eq}\big(\varphi^2\big)$.
\end{Lemma}

Since under the background assumption the shortest arrows in
$\tGamma(\varphi)$ and $\Gamma_{\eq}\big(\varphi^2\big)$ are unique,
this will establish the theorem.

\begin{proof}[Proof of Lemma \ref{lemma:key}] In the notation from
  Section \ref{sec:eq-Floer homology}, set
  \begin{align*}
    \halpha_i & = \wp(\alpha_i\otimes\alpha_i),\\
    \heta_j   & =
                \hh\wp(\eta_j\otimes\eta_j) +
                \wp(\eta_j\otimes\gamma_j),\\
    \hgamma_j & = \wp(\gamma_j\otimes\gamma_j).
  \end{align*}
  Then, by Seidel's non-vanishing theorem (Proposition
  \ref{prop:non-van}),
  \[
    \heta_1=\hh^m\bx_*^2+\ldots\textrm{ and }
    \hgamma_1=\hh^{m'}\by_*^2+\ldots
  \]
  for some $m\geq 0$ and $m'\geq 0$, where the dots again stand for
  higher action terms.

  Since $\wp$ is a chain map, i.e.,
  $\wp\circ d_{\Z_2}=d_{\eq}\circ\wp$, we have
  \[
    d_{\eq}\halpha_i=0
  \]
  and
  \begin{align*}
  d_{\eq}\heta_j
    &=\hh\wp(\gamma_j\otimes\eta_j)+\hh\wp(\eta_j\otimes\gamma_j)\\
    &\quad + \wp(\hh\eta_j\otimes\gamma_j+\hh\gamma_j\otimes\eta_j)\\
    &\quad + \wp(\gamma_j\otimes\gamma_j)\\
    &=\hgamma_j.
  \end{align*}

  This indicates that the collection
  $\hCB:=\{\halpha_i,\,\heta_j,\,\hgamma_j\}$ can be thought of as a
  singular decomposition of $\CF_{\eq}\big(\varphi^2\big)$ with the
  minimal bar given by
  \[
    \CA(\hgamma_1)-\CA(\heta_1)
    =\CA\big(\by_*^2\big)-\CA\big(\bx_*^2\big),
  \]
  and, arguing similarly to Step 1, we should be able to show that
  $\bx_*^2$ and $\by_*^2$ are connected by the shortest arrow. A minor
  technical difficulty that arises at this stage is that
  $\CF_{\eq}\big(\varphi^2\big)$ does not fit in with the algebraic
  framework of Section \ref{sec:barcodes} or \cite{UZ}. Namely,
  $\CF_{\eq}\big(\varphi^2\big)$ is not finite-dimensional over
  $\Lambda$; it is finite-dimensional over $\Lambda[\hh]$, but the
  latter is not a field. We circumvent this difficulty by a trick
  which essentially amounts to setting $\hh=1$. (This is the point
  where our choice of working with polynomials in $\hh$ rather than
  formal power series as in \cite{Se} is essential; cf.\
  Remark~\ref{rmk:pol}.)

  Consider the ungraded complex $\tCC$ defined as follows:
  $\tCC:=\CF\big(\varphi^2\big)\subset \CF_{\eq}\big(\varphi^2\big)$
  as a vector space over $\Lambda$ with the differential
  $\td\alpha:=d_{\eq}\alpha|_{\hh=1}$ for $\alpha\in \tCC$. Since
  $d_{\eq}$ is $\hh$-linear, we have $\td^2=0$. More formally, $\tCC$
  is the quotient complex in the short exact sequence of ungraded
  complexes
  \[
    0\longrightarrow
    \CF_{\eq}\big(\varphi^2\big)\stackrel{1+\hh}{\longrightarrow}
    \CF_{\eq}\big(\varphi^2\big)
    \stackrel{\pi}{\longrightarrow}\tCC\longrightarrow 0
  \]
  over $\Lambda$, where $\pi$ is the $\hh=1$ evaluation map.

\begin{Remark}
    \label{rmk:HtC}
    This exact sequence, for any action interval, gives rise to the
    exact triangle in Floer cohomology relating $\H(\tCC)$ and
    $\HF_{\eq}\big(\varphi^2\big)$ via multiplication by $1+\hh$. As
    any map of the form $\id + O(\hh)$, this multiplication map in
    Floer cohomology is one-to-one, and thus
    \[
      \H(\tCC)\cong \HF_{\eq}\big(\varphi^2\big)/(1+\hh)
      \HF_{\eq}\big(\varphi^2\big),
    \]
    and hence
    $\dim_{\F_2}\H(\tCC)=\rk_{\F_2[\hh]}\HF_{\eq}\big(\varphi^2\big)$,
    for any action interval. For global cohomology,
    $\H(\tCC)\cong \HF\big(\varphi^2\big)$ as ungraded
    $\Lambda$-modules by the continuation argument and Example
    \ref{ex:C2small}.
  \end{Remark}
  
  Since, by construction, $\tCC$ is a finite-dimensional vector space
  over $\Lambda$, now the machinery from \cite{UZ} applies literally;
  see Remark \ref{rmk:algebra}. In self-explanatory notation,
  \[
    \left<d_{\eq} \bz,\hh^m\bz'\right>\neq 0 \textrm{ where }
    m=\mu(\bz)-\mu(\bz')+1 \Longleftrightarrow \big<\td
    \bz,\bz'\big>\neq 0
  \]
  for $\bz$ and $\bz'$ in $\bPP_2(\varphi)$. Furthermore, we can also
  form the Floer graph for $\tCC$ and this graph is identical to the
  equivariant Floer graph $\Gamma_{\eq}\big(\varphi^2\big)$.
  
  \begin{Claim}
    \label{claim}
    The subset $\tCB:=\pi(\hCB)$ in $\tCC$ formed by
    $\talpha_i:=\pi(\halpha_i)$ and $\teta_j:=\pi(\heta_j)$ and
    $\tgamma_j:=\pi(\hgamma_j)$ is a singular decomposition for
    $\tCC$.
  \end{Claim}

  Putting aside the proof of the claim, let us first show how Lemma
  \ref{lemma:key} follows from it.  Observe that
  \begin{equation}
    \label{eq:2action}
    \CA(\tgamma_j)-\CA(\teta_j)=2\big(\CA( \gamma_j)-\CA(\eta_j) \big).
  \end{equation}
  Indeed, set
  \begin{align*}
    \eta_j &=\bx_j+\ldots,\\
    \gamma_j &=\by_j+\ldots,
  \end{align*}            
  where as usual the dots stand for strictly higher action
  terms. (Thus $\bx_*=\bx_1$ and $\by_*=\by_1$.) By Seidel's
  non-vanishing theorem (Proposition \ref{prop:non-van}), we have
    \begin{align*}
    \heta_j &=\hh^{m_j}\bx_j^2+\ldots,\\
    \hgamma_j &=\hh^{m'_j}\by_j^2+\ldots
    \end{align*}
    for some $m_j\geq 0$ and $m'_j\geq 0$, and hence
    \begin{align*}
    \teta_j &=\bx_j^2+\ldots,\\
    \tgamma_j &=\by_j^2+\ldots.
    \end{align*}
    Therefore,
     \[
       \CA(\tgamma_j)-\CA(\teta_j) =
       \CA\big(\by_j^2\big)-\CA\big(\bx_j^2\big) =
       2\big(\CA(\by_j)-\CA(\bx_j)\big) = 2\big(\CA(
       \gamma_j)-\CA(\eta_j) \big),
     \]
     which proves \eqref{eq:2action}.

     In particular, similarly to \eqref{eq:increasing_seq}, we have
    \[
      \CA(\tgamma_1)-\CA(\teta_1)<\CA(\tgamma_2)-\CA(\teta_2)<\ldots .
    \]
    Therefore,
    \[
      \beta_{\min}(\tCC):=\CA(\tgamma_1)-\CA(\teta_1)
      =\CA\big(\by_*^2\big)-\CA\big(\bx_*^2\big)
    \]
    is the shortest bar for $\tCC$. As in Step 1, we infer that
    \[
      \big<\td\bx_*^2,\by_*^2\big>=1.
    \]
    Hence there is an arrow connecting these two orbits in the Floer
    graph for $\tCC$ and this is the shortest arrow. The Floer graph
    for $\tCC$ is defined similarly and in fact identical to the
    equivariant Floer graph
    $\Gamma_{\eq}\big(\varphi^2\big)$. Therefore, this arrow is also
    the shortest arrow in $\Gamma_{\eq}\big(\varphi^2\big)$,
    completing the proof of Lemma \ref{lemma:key} modulo Claim
    \ref{claim}.

    \begin{proof}[Proof of Claim \ref{claim}] Since $\pi$ is a
      homomorphism of complexes, we have $\td\talpha_i=0$ and
      $\td\teta_j=\tgamma_j$. Therefore, we only need to show that
      $\tCB$ is an orthogonal basis. For this we do not need to
      distinguish between different types of elements of $\CB$.  Write
      $\CB=\{\xi_i\}$, where $\xi_i=\bz_i+\ldots$ with the dots
      denoting the entries of strictly higher action. Then, by the
      definition of $\hCB$ and Seidel's non-vanishing theorem,
      $\tCB=\{\txi_i\}$ comprises the elements
    \[
      \txi_i:=\pi(\hxi_i)=\bz_i^2+\ldots.
    \]

    Now, as in Example \ref{exam:orthogonality}, the orthogonality
    for $\CB$ is equivalent to that the orbits $\bz_i$ are
    distinct. Similarly, the orthogonality for $\tCB$ is equivalent to
    that the orbits $\bz_i^2$ are again distinct. It follows that
    $\tCB$ is orthogonal if (in fact, iff) $\CB$ is orthogonal which
    is a part of its definition. As a consequence, $\tCB$ is linearly
    independent over~$\Lambda$.

    Finally, since $\tCC=\CF\big(\varphi^2\big)$ as $\Lambda$-modules
    and $\varphi$ is 2-perfect, we have
    \[
      \dim_\Lambda\tCC = \dim_\Lambda \CF\big(\varphi^2\big) =
      \dim_\Lambda \CF(\varphi)=|\CB|=|\tCB|,
    \]
    and $\tCB$ is a basis.
  \end{proof}

  This concludes the proof of Lemma \ref{lemma:key}.
\end{proof}

\medskip\noindent\emph{Step 3: Removing the background assumption.}
Recall that the Floer graphs $\Gamma(\varphi)$ and
$\Gamma_{\eq}\big(\varphi^2\big)$ are stable under small perturbations
of $\varphi$. With this in mind, we can replace $\varphi$ by a
$C^\infty$-small perturbation $\tvarphi$ meeting the background
assumption, since the latter is a $C^\infty$-generic condition. More
precisely, one can change the action of a single orbit by a small
amount (positive or negative) using a localized $C^\infty$-small
perturbation $\tvarphi$. Hence, given any arrow in the Floer graphs
$\tGamma(\varphi)$ and $\tGamma_{\eq}\big(\varphi^2\big)$, pick some
small $\epsilon>0$. Then one can apply local perturbations at the two
ends to shorten its length by $2\epsilon$ while not changing the
lengths of the remaining arrows more than $\epsilon$. It follows that
every shortest arrow in the Floer graphs $\tGamma(\varphi)$ and
$\tGamma_{\eq}\big(\varphi^2\big)$ can be perturbed into the unique
shortest arrow. Now, Theorem \ref{thm:main} for $\varphi$ follows from
that theorem for $\tvarphi$. \qed

\begin{Remark}[The $\Z_p$-equivariant analogue] 
  \label{rmk:char-p2}
  This argument extends with only very minor changes to the $p$th
  iterates $\varphi^{p}$, where $p$ is a prime, proving the analogue
  of Theorem \ref{thm:main} for $\Z_p$-equivariant cohomology of
  $\varphi^p$ over $\F_p$ and relying on the results from \cite{ShZ};
  cf.\ Remark \ref{rmk:char-p}. As a consequence, as in the proof of
  Theorem \ref{thm:Sh}, if $\varphi$ is strongly non-degenerate,
  $\beta$ is \emph{a priori} bounded from above and $|\PP(\varphi)|$
  is greater than the sum of Betti numbers of $M$ over $\Q$, then
  there exists a simple $p$-periodic orbit for every sufficiently
  large prime $p$ as is shown in \cite{Sh:HZ}.
\end{Remark}

\subsection{Degenerate case}
  \label{sec:degenerate}
  Perhaps, the simplest way to extend our arguments and, in
  particular, Theorem \ref{thm:Sh} and Corollary \ref{cor:Sh} to
  include some degenerate Hamiltonian diffeomorphisms as in
  \cite{Sh:HZ} is by bypassing Theorem \ref{thm:main} and using a
  somewhat less precise argument. Below we outline the key steps of
  this generalization, some of which again overlap with \cite{Sh:HZ}.
  The account is deliberately brief. The main new point here is the
  construction of the (equivariant) Floer graph in the degenerate
  case.

  Assume that $\varphi$ is 2-perfect and that the second iteration is
  admissible: $-1$ is not an eigenvalue of $D\varphi_x$ for any
  $x\in\PP_1(\varphi)$. (The latter requirement is satisfied once
  $\varphi$ is replaced by its sufficiently high iterate
  $\varphi^{2^k}$.) Then, as shown in \cite{GG:gap}, for every
  $\bx\in\bPP_1(\varphi)$ we have a canonical isomorphism in local
  Floer cohomology:
  \begin{equation}
    \label{eq:local}
    \HF(\bx)\stackrel{\cong}{\longrightarrow}\HF\big(\bx^2\big)
  \end{equation}
  up to a shift of grading. By the Smith inequality in local Floer
  cohomology, which can be proved by exactly the same argument as in
  \cite{Se} (see also \cite{CG,Sh:HZ}), we have
  $\HF_{\eq}\big(\bx^2\big)\cong \HF\big(\bx^2\big)[\hh]$, where,
  strictly speaking, on the left we have the graded module associated
  with the $\hh$-adic filtration of $\HF_{\eq}\big(\bx^2\big)$. (We
  expect that in this situation $d_{\eq}=d_{\Fl}$, and hence
  $\HF_{\eq}\big(\bx^2\big)\cong \HF\big(\bx^2\big)[\hh]$ literally,
  without passing to graded modules, but we have not been able to
  prove this.)

  For every $\bx\in\bPP_1(\varphi)$, fix a basis $\xi_{i,\bx}$ in
  $\HF(\bx)$ so that this system of bases is
  recapping-invariant. Applying \eqref{eq:local} to this system, we
  obtain bases $\xi'_{i,\bx}$ in $\HF\big(\bx^2\big)$ with
  $\bx\in\bPP_1(\varphi)$, and this system extends to a
  recapping-invariant system over the entire $\bPP_2(\varphi)$.

  We also have a recapping-invariant system of bases in
  $\HF_{\eq}\big(\bx^2\big)$ arising from
  $\wp(\xi_{i,\bx}\otimes \xi_{i,\bx})\in\HF_{\eq}\big(\bx^2\big)$. To
  be more precise, it is convenient to replace the equivariant
  cohomology (local or global) by the homology of the ungraded complex
  $\tCC$ obtained by setting $\hh=1$ as in the proof of Theorem
  \ref{thm:main}. For the sake of brevity, we keep the notation
  $\HF_{\eq}$ for this cohomology suppressing the projection $\pi$ in
  the notation. Set
  $\xi^{\eq}_{\bx,i}:=\wp(\xi_{i,\bx}\otimes \xi_{i,\bx})$. We claim
  that this is a basis in $\HF_{\eq}\big(\bx^2\big)$ which is now just
  a vector space over $\F_2$. Then, extending, we get a recapping
  invariant family of bases over $\bPP_2(\varphi)$.

  To show that $\{\xi^{\eq}_{\bx,i}\}$ is indeed a basis, we first
  recall that, without changing $D\varphi_x$ and the local cohomology,
  $\varphi$ can be deformed near $x$ to the direct product of
  degenerate and totally non-degenerate maps; see \cite[Sect.\
  4.5]{GG:gap}. This essentially reduces the question to the case,
  which for the sake of brevity we will focus on, where $x$ is totally
  degenerate, i.e., all eigenvalues of $D\varphi_x$ are equal to 1 and
  in particular $\varphi$ can be made $C^1$-close to the
  identity. Furthermore, recall that
  $\HF(\bx)\cong\HF(\varphi_f)\cong\HM(f)$ by \cite[Sect.\ 3.3 and
  6]{Gi:CC}, where $\HM$ stands for the local Morse cohomology, $f$ is
  the generating function of $\varphi$ and $\varphi_f$ is the germ of
  the Hamiltonian diffeomorphism generated by $f$. These isomorphisms
  come from continuation maps and there are similar isomorphisms
  (equivariant and non-equivariant) for $\bx^2$ and
  $\varphi_{2f}=\varphi_f^2$, where we can replace the generating
  function for $\varphi^2$ by $2f$; see \cite[Sect.
  4.3]{GG:gap}. Now, as in Example \ref{ex:C2small} and Remark
  \ref{rmk:HtC}, we arrive at the continuation map identifications
  \begin{equation}
    \label{eq:CI}
  \HF_{\eq}\big(\bx^2\big)\cong\HF\big(\bx^2\big)\cong\HF(\bx)\cong
  \H(Y_f),
\end{equation}
where $Y_f$ is a certain topological space (the Conley index)
associated with the critical point $x$ of $f$. Furthermore, the map
$\alpha\mapsto \wp(\alpha\otimes \alpha)$ turns into the Steenrod
square $\Sq$ on $\H(Y_f)$; see \cite{Wi1}. Thus, with these
identifications in mind, $\xi_{\bx,i}=\xi'_{\bx,i}$ and
\begin{equation}
  \label{eq:Sq}
  \xi_{\bx,i}^{\eq}=\Sq(\xi_{\bx,i})=\xi_{\bx,i}+\ldots,
\end{equation}
where the dots stand for the terms of \emph{higher degree} in
$\H(Y_f)$. It follows that the vectors $\xi_{\bx,i}^{\eq}$ are
linearly independent and, since
$\dim_{\F_2} \HF_{\eq}\big(\bx^2\big)=\dim_{\F_2}\HF(\bx)$ by
\eqref{eq:CI}, this system is a basis.

The action filtration spectral sequence in Floer cohomology has
$E_1=\bigoplus_{\bx}\HF(\bx)$ and converges to $\HF(\varphi)$. With
bases fixed, we can canonically collapse this spectral sequence into
one complex with the same features as the ordinary Floer complex
including the action filtration and cohomology equal to
$\HF(\varphi)$; cf.\ \cite[Sect.\ 2.1.3 and 2.5]{GG:LS}. This data is
sufficient to define the Floer graph $\Gamma(\varphi)$ of $\varphi$
with vertices $\xi_{\bx,i}$. (Note that the orbits with $\HF(\bx)=0$
do not contribute to $\Gamma(\varphi)$ and the graph depends on the
choice of the bases $\{\xi_{\bx,i}\}$.) It is also worth keeping in
mind that even in the non-degenerate case this graph and the complex
might differ from the Floer graph as defined in Section
\ref{sec:Floer_Gr} and from the Floer complex. However, they have the
same formal properties as $\CF(\varphi)$ and the original graph, and
the resulting homology is isomorphic to the Floer cohomology
$\HF(\varphi)$; cf.\ \cite{GG:LS}.

A similar construction applies to $\varphi^2$ in the ordinary and
equivariant settings and
$\xi'_{\bx,i}\leftrightarrow \xi_{\bx,i}^{\eq}$ gives rise to an
action-preserving one-to-one correspondence between the vertices of
$\Gamma\big(\varphi^2\big)$ and $\Gamma_{\eq}\big(\varphi^2\big)$.
The condition that the sum \eqref{eq:local-sum} with $\F=\F_2$ is
strictly greater than the sum of Betti numbers guarantees that the
graph $\Gamma(\varphi)$, and hence $\Gamma\big(\varphi^2\big)$ and
$\Gamma_{\eq}\big(\varphi^2\big)$, have at least one arrow.
    
Denote by $\beta_{\min}$ the length of the shortest arrows in a Floer
graph. Our goal is to show that $\varphi$ cannot be $2^k$-perfect,
where $k$ is sufficiently large, assuming an \emph{a priori} upper
bound on $\beta_{\min}\big(\varphi^{2^k}\big)$ as in Theorem
\ref{thm:Sh}. (Note that in contrast with the non-degenerate case the
Floer graphs are now sensitive to small perturbations of $\varphi$ and
we usually cannot make the shortest arrow unique without changing the
graph unless $\dim_{\F_2}\HF(x)=1$ for all $x\in\PP_1(\varphi)$.)

The equivariant pair-of-pants product $\wp$ extends to the complexes
we have constructed, and Seidel's non-vanishing theorem takes the form
  \begin{equation}
    \label{eq:Se-nonvanishing2}
    \wp(\xi_{\bx,i}\otimes\xi_{\bx,i})=\xi_{\bx,i}^{\eq}+\ldots,
\end{equation}
where now the dots stand for terms with action greater than or equal
to the action of $\xi_{\bx,i}^{\eq}$, but with the provision that the
first term enters the whole sum with non-zero coefficient. (This is a
consequence of \eqref{eq:Sq} and Seidel's non-vanishing theorem
applied to the non-degenerate part in the splitting of $\varphi$ at
$x$.)
  
Pick one of the shortest arrows, say $v$, in
$\Gamma_{\eq}\big(\varphi^2\big)$. After recapping, we can ensure that
the beginning of $v$ has the form $\xi^{\eq}_{\bx,i}$. Using
\eqref{eq:Se-nonvanishing2} and the facts that $\wp$ is a chain map
and $v$ is a shortest arrow, it is not hard to see that $\xi_{\bx,i}$
is the beginning of an arrow in $\Gamma(\varphi)$ whose length is at
most $\beta_{\min}^{\eq}\big(\varphi^2\big)/2$. Hence,
  \begin{equation}
    \label{eq:inequality2}
    2\beta_{\min}(\varphi)\leq \beta_{\min}^{\eq}\big(\varphi^2\big).
  \end{equation}
  (This proves a somewhat weaker version of Theorem \ref{thm:main}:
  every shortest equivariant arrow comes from an arrow for $\varphi$.)

On the other hand,  
  \begin{equation}
\label{eq:inequality1}
\beta_{\min}^{\eq}\big(\varphi^2\big)\leq
\beta_{\min}\big(\varphi^2\big).
  \end{equation}
  Indeed,
  $\dim_{\F_2}\HF^I\big(\varphi^2\big) \geq
  \rk_{\F_2[\hh]}\HF^I_{\eq}\big(\varphi^2\big)$ for any action
  interval $I$, as is easy to see from the $\hh$-adic filtration
  spectral sequence.  Applying this to an interval tightly enclosing
  one of the shortest arrows in $\Gamma_{\eq}\big(\varphi^2\big)$ we
  obtain \eqref{eq:inequality1}. In fact, we expect that, as in the
  non-degenerate case, $\Gamma_{\eq}\big(\varphi^2\big)$ incorporates
  all arrows of $\Gamma\big(\varphi^2\big)$ (and, perhaps, more). This
  is a stronger statement than \eqref{eq:inequality1}, but
  \eqref{eq:inequality1} is sufficient for our purposes.
  
  Combining \eqref{eq:inequality2} and \eqref{eq:inequality1}, we see
  that $2\beta_{\min}(\varphi)\leq \beta_{\min}\big(\varphi^2\big)$.
  As a consequence,
  $\beta_{\min}\big(\varphi^{2^k}\big)\geq 2^k\beta_{\min}(\varphi)$
  as long as $\varphi$ is $2^k$-perfect. When
  $\beta_{\min}\big(\varphi^{2^k}\big)$ is bounded from above, this is
  impossible for large $k$.

  We note in conclusion that in the non-degenerate case this proof
  reduces to an argument which does not rely on persistence homology
  and is ultimately simpler and more direct, although arguably less
  structured, than our proof of Theorem \ref{thm:Sh} via Theorem
  \ref{thm:main}.

\end{document}